\documentclass[12pt, a4paper]{article}
\usepackage{amsfonts}
\usepackage{amssymb}
\usepackage{latexsym}
\usepackage{mathrsfs}
\usepackage{tikz}
\usepackage[all]{xy}
\usepackage{graphicx}
\usepackage{multicol}
\usepackage{amsmath}
\usepackage{relsize}
\usetikzlibrary{matrix,arrows,decorations.pathmorphing}
\usepackage{amsmath,amsthm}
\usepackage{float}
\floatstyle{boxed} 
\restylefloat{figure}
\usepackage[usenames,dvipsnames]{pstricks}
\usepackage{epsfig}
\usepackage{pst-grad} 
\usepackage{pst-plot} 
\usepackage[space]{grffile} 
\usepackage{etoolbox}
\usepackage{epstopdf}
\usepackage{amssymb}
\usepackage{latexsym}
\usepackage{color}

\usepackage{mathrsfs}
\usepackage{pst-grad} 
\usepackage{pst-plot}
\usepackage[latin1]{inputenc}
\usepackage{psfrag}
\usepackage{graphics}
\usepackage{graphicx}
\usepackage[usenames,dvipsnames]{pstricks}
\usepackage{pst-grad} 
\usepackage{pst-plot}
\usepackage{mathbbol}
\usepackage[latin1]{inputenc}
\usepackage{psfrag}
\usepackage{graphics,graphicx}
\usepackage[english,  activeacute]{babel}
\usepackage{amssymb}
\usepackage{amsthm}
\usepackage{array}
\usepackage{pdflscape}
\usepackage{a4wide}
\usepackage{color, url}
\usepackage{float}
\theoremstyle{plain}
\newtheorem{theorem}{Theorem}
\newtheorem{proposition}{Proposition}
\newtheorem{corollary}[theorem]{Corollary}
\newtheorem{remark}{Remark}

\theoremstyle{definition}
\newtheorem{definition}[theorem]{Definition}
\newtheorem{example}[theorem]{Example}

\newcommand{\C}{\mathscr{C}}

\newcommand{\pa}{\mathcal{P}}
\newcommand{\ra}{\rightarrow}
\newcommand{\raa}{\ran\ran}
\newcommand{\laa}{\langle\langle}
\newcommand{\A}{{\mathbb A}}
\newcommand{\N}{{\mathbb N}}

\newcommand{\X}{{\mathbb X}}

\newcommand{\Al}{\mathbb{\Sigma}}
\newcommand{\Lo}{\mathbb{L}}
\newcommand{\Lsl}{^{\Al_0}\mathbb{L}_+}
\newcommand{\sh}{\sigma}
\newcommand{\Lb}{\mathcal{P}}
\newcommand{\Nb}{\mathcal{N}}
\newcommand{\la}{\langle}
\newcommand{\ran}{\rangle}

\newcommand{\lam}{\lambda}
\newcommand{\om}{\omega}
\newcommand{\msA}{\mathscr{A}}

\newcommand{\spl}{\circ_s}

\newcommand{\kap}{\kappa}

\newcommand{\K}{\mathbb{K}}

\providecommand{\keywords}[1]
{
	\small	
	\textbf{\textit{Keywords---}} #1
}
\providecommand{\subjclass}[1]
{
	\small	
	\textbf{\textit{ Subject class---}} #1
}

\title{Abstract, keywords and references template}
\author{Miguel A. M\'endez\\
[-0.8ex]\small Universidad Yachay\\
	[-0.8ex]\small School of Mathematical\\[-0.8ex]\small \& Computational Sciences \\
[-0.8ex]\small Urcuqu\'i, Ecuador.  
\\[-0.8ex]\small\tt mmendezenator@gmail.com}

\title{Shift-Plethystic Trees and Rogers-Ramanujan Identitites.}

\begin{document}

\begin{abstract}By studying non-commutative series in an infinite alphabet we introduce shift-plethystic trees and a class of integer compositions as new combinatorial models for the Rogers-Ramanujan identities. We prove that the language associated to shift-plethystic trees can be expressed as a non-commutative generalization of the Rogers-Ramanujan continued fraction. By specializing the noncommutative series to $q$-series we obtain new combinatorial interpretations to the Rogers-Ramanujan identities in terms of signed integer compositions. We introduce the operation of shift-plethysm on non-commutative series and use this to obtain interesting enumerative identities involving compositions and partitions related to Rogers-Ramanujan identities. 
\end{abstract}

\maketitle
\\
\noindent\keywords{Rogers-Ramanujan identities, Integer Compositions, Non-commutative series}\\

\section{Introduction}
The Rogers-Ramanujan identities, equations  (\ref{eq.RR2}) and (\ref{eq.RR1}), have had a fructiferous influence in many, some of them unexpected, subjects in Mathematics and Physics.

\begin{eqnarray}\label{eq.RR2}\sum_{n=0}^{\infty}\frac{ q^{n^2}}{(1-q)(1-q^2)\dots (1-q^n)}&=&\prod_{k=0}^{\infty}\frac{1}{(1-q^{5k+1})(1-q^{5k+4})}\\\label{eq.RR1}
\sum_{n=0}^{\infty}\frac{q^{n(n+1)}}{(1-q)(1-q^2)\dots (1-q^n)}&=&\prod_{k=0}^{\infty}\frac{1}{(1-q^{5k+2})(1-q^{5k+3})}
\end{eqnarray}
 They were discovered and proved by Rogers in 1894 \cite{rogers1894second}, rediscovered by Ramanujan (without proof) in 1913, and again by I. Schur in 1917   \cite{Schur1917}. It is impossible to summarize in a few lines the enormous amount of contributions related  to the Rogers-Ramanujan identities and their generalizations.
 The reader is referred to the recent book of Sills \cite{sills2017invitation}, for further references and a nice introduction to the subject in its historical context. 
In \cite{rogers1894second} Rogers presented what is now known as the Rogers-Ramanujan continued fraction $\mathcal{R}(q)$,

$$\mathcal{R}(q)=q^{\frac{1}{5}}\frac{1}{1+\cfrac{q}{1+\cfrac{q^2}{\ddots}}}$$
 and proved that $$\mathcal{R}(q)=q^{\frac{1}{5}}\frac{\sum_{n=0}^{\infty}\frac{q^{n(n+1)}}{(1-q)(1-q^2)\dots (1-q^n)}}{\sum_{n=0}^{\infty}\frac{ q^{n^2}}{(1-q)(1-q^2)\dots (1-q^n)}}.$$

In what follows we shall drop the factor $q^
{\frac{1}{5}}$ from $\mathcal{R}(q)$, since our main concern here is about the combinatorial meaning of the Rogers-Ramanujan continued fraction and identities. MacMahon \cite{MacMahonbook} and Schur \cite{Schur1917} were the first in reporting the combinatorial meaning of the Rogers-Ramanujan identities.  The left-hand side of (\ref{eq.RR2}) is the generating function for the number of partitions of positive parts with a difference of at least two among adjacent parts ($2$-distinct partitions in the terminology of \cite{Andrews2004}). Its right hand side is the generating function of the partitions with each part congruent either with one or four module five. Similarly, the left-hand side of (\ref{eq.RR1}) is the generating function for the number of $2$-distinct partitions, but having each part strictly greater than one. Its right hand side counts the number of partitions with each part congruent either with two or three module five.  
Hence, each of them establish an equipotence between two different sets of partitions. Garsia and Milne gave a  bijective proof of the Rogers-Ramanujan identities by establishing  a complicated bijection between these two kinds of partitions \cite{garsia1981rogers}. For that end they created what now is called the Garsia-Milne involution principle. The Garsia-Milne proof was later simplified in \cite{bressoud1982short}.\\
We introduce here a noncommutative version of $\mathcal{R}(-q)$ in an infinite number of variables $$X_0, X_1, X_2,X_3,\dots,$$ and prove that its expansion is the language of words associated to a combinatorial structure we call shift-plethystic trees.         
 Our model based on shift-plethystic trees lead us to consider compositions (instead of partitions) whose risings are at most one,  and express the non-commutative version of $\mathcal{R}(-q)$ as a quotient of two generating functions on this kind of compositions. We call a $q$-umbral evaluation on a noncommutative series the procedure of substituting each variable $X_k$ by $zq^k$ or simply by $q^k$. By $q$-umbral evaluation of those generating functions we obtain an alternative (dual) combinatorial interpretation of Rogers-Ramanujan identities in terms of signed compositions (Section \ref{sec.RRcompositions}). A combinatorial understanding of the cancellations taking place in the signed compositions that we obtain would provide an elegant and, hopefully, simple proof of the Rogers-Ramanujan identities.  In Section \ref{sec.splety} we introduce  shift plethysm of non-commutative series. It generalizes the classical substitution of $q$ series. By $q$-umbral evaluating shift-plethysm on a particular class of non-commutative series we obtain the classical substitution of $q$-series. By means of elementary computation of inverses on generalized shift-plethystic trees we recover some classical identities in Subsection \ref{sec.pletitrees}, and prove in Section \ref{sec.RRandnew} new ones relating Rogers-Ramanujan identities, compositions, partitions and shift plethystic trees.
 
Previous work on non-commutative versions of the Rogers-Ramanujan continued fractions can be found  in \cite{Berenstein2019} and \cite{Pak}. Although their approach does not rely on an infinite number of variables, a coupling of both approaches would lead to novel identities involving signed compositions.

\section{Formal power series in non commuting variables}
Let $\mathbb{A}$ be be an alphabet (a totally ordered set) with at most a countable number of elements (letters). Let $\A^*$ be the free monoid generated by $\A$. It consists of words or finite strings of letters in $\A$, $\om=\om_1 \om_2\dots \om_n$, including the empty string represented as $1$. We denote by $\ell(\om)$ the length of $\om$.
Let $\K$ be a field of characteristic zero. A \emph{noncommutative} formal power series in $\A$ over $\K$ is a function $R:\A^*\rightarrow \K$. We denote $R(\om)$ by $\la R,\om\ran$ and represent $R$ as a formal series

$$R=\sum_{\om\in\mathbb{A}^*}\langle R,\om\rangle\,\om, \;\langle R,\om\rangle\in\mathbb{K},$$
 
The sum and product of two formal power series $R$ and $S$ are respectively given by 
\begin{eqnarray*}
	R+S&=&\sum_{\om\in\A^*}(\la R,\om\ran+\la S,\om\ran) \om\\R.S&=&\sum_{\om\in \mathbb{A}^*}(\sum_{\om_1\om_2=\om}\langle R,\om_1\rangle \langle S,\om_2\rangle) \omega. \end{eqnarray*}
The algebra of noncommutative formal power series is denoted by $\K\la\la\A\ran\ran$.
There is a notion of convergence on $\K\la\la\A\ran\ran$. We say that $R_1, R_2, R_3,\dots$ converges to $R$ if for all $\om\in \A^*$, $\la R_n,\om\ran= \la R,\om\ran$ for $n$ big enough. If $\la R,1\ran=\alpha\neq 0$, then $R$ has an inverse given by (see Stanley \cite{stanley5001enumerative})
$$R^{-1}=\frac{1}{\alpha}\sum_{n=0}^{\infty}\left(1-\frac{R}{\alpha}\right)^n.$$

Let $B$ be a series having constant term equal to zero, $\la B,1\ran =0$. We denote by $\frac{1}{1-B}$, the inverse of the series $1-B$,
$$\frac{1}{1-B}:=(1-B)^{-1}=\sum_{n=0}^\infty B^n.$$

A {\em\textcolor{blue} {language}} (on $\A$) is a subset of $\A^*$. We identify a language $L$ with its generating function, the formal power series
$$L=\sum_{\om\in L} \om.$$
We consider now a special kind of languages obtained from a given set of `links' $B\subseteq\A\times\A$. Define
$$L_B=\{\om|(\om_i,\om_{i+1})\in B, \mbox{ for every }i=1,2,\dots,\ell(\om)-1\},$$
and the language $L$ associated to $B$ by
\begin{equation}\label{eq.linked}L=1+\A+L_B.
\end{equation}
We shall call an $L$ of this form a {\em \textcolor{blue}{linked language}}. 
Define the K-dual $L^!$ to be the language associated with the complement set of links
$$L^!=1+\A+L_{B^c}$$
For linked languages we define a second formal power series, 
$$L^{g}=\sum_{\om\in L}(-1)^{\ell(\om)}\om.$$ 
We call it the {\em \textcolor{blue}{graded}} generating function of $L$. 
We have the following inversion formula for linked languages. It is a non-commutative version of  Theorem 4.1. in Gessel PhD thesis, \cite{Gesselthesis}, from where we borrow the terminology of linked sets. Propositions \ref{prop.kdual1} and \ref{prop.kdual2} are indeed particular instances of inversion formulas on generating functions for Koszul algebras and Koszul modules over Koszul algebras. Koszul algebras were introduced in \cite{priddy1970koszul},  see also \cite{polishchuk2005quadratic} for more details on Koszul algebras and the inversion formulas for generating functions of Koszul algebras and modules.
	
\begin{proposition}\label{prop.kdual1}
	Let $L$ be a linked language, and $L^!$ its K-dual. Then we have 
	\begin{equation}
	L^!=(L^g)^{-1}.
	\end{equation}
\end{proposition}
\begin{proof}
The product $L^g.L^!$ is equal to
\begin{equation}
L^g.L^!=\sum_{(\om,\om')\in L\times L^! }(-1)^{\ell(\om')}\om.\om'.
\end{equation}
Define the function $\phi:L\times L^!\ra L\times L^!$,
$$\phi(\om_1\om_2\dots\om_k,\om'_1\om'_2\dots\om'_j)=(\om_1\om_2\dots\om_k\om'_1,\om'_2\dots\om'_j)$$ when $(\om_k,\om'_1)\in B$ or if $\om=1$ and $\om'\neq 1.$ Make  
$$\phi(\om_1\om_2\dots\om_k,\om'_1\om'_2\dots\om'_j)=(\om_1\om_2\dots,\om_k\om'_1\om'_2\dots\om'_j)$$ if $(\om_k,\om'_1)\in B^c$ or if $\om'=1$ and $\om\neq 1$. Finally, make $\phi(1,1)=(1,1)$. The function $\phi$ is a sign reversing involution when restricted to the signed set $L\times L^!-\{(1,1)\}$. Moreover  $\phi(1,1)=(1,1)$. Hence
$L^!L^g=1.1=1$.
\end{proof}
\begin{example}
	Let $\X_+$ be the infinite alphabet $\{X_1,X_2,X_3,\dots\}$. Denote by $\pa$ the set of partitions $\lambda$ written in weak increasing order, $\lam_1\leq\lam_2\leq\lam_3\dots$. The set $\pa$ is represented as a language with letters in $\X_+$,
	\begin{equation*}
	\pa=\sum_{\lam \in }X_{\lam}=\lim_{m\rightarrow \infty}\prod_{n=1}^{m}\frac{1}{1-X_n}=\prod_{n=1}^{\infty}\frac{1}{1-X_n}.
	\end{equation*}
	It is a linked language, with set of links $B=\{(X_i,X_j)|i\leq j\}$. The complement is the set $B^c=\{(X_i,X_j)|i>j\}$, and hence the $K$-dual language $\pa^!$ is the generating functions of the set of partitions in decreasing order with distinct parts.
	The graded generating function of $\pa$ is equal to 
	\begin{equation*}
	\pa^g=\sum_{\lam \in \pa}(-1)^{\ell(\lam)}X_{\lam}=\lim_{m\rightarrow\infty}\prod_{n=1}^{m}\frac{1}{1+X_n}=\prod_{n=1}^{\infty}\frac{1}{1+X_n}
	\end{equation*}
	By Proposition \ref{prop.kdual1}, since taking inverses is a continuous operation,  
	\begin{equation*}
	\pa^!=(\pa^g)^{-1}=\lim_{m\rightarrow\infty}(1+X_m)(1+X_{m-1})\dots(1+X_1)
	\end{equation*}
	This limit can be (symbolically) written  as the product
	$\prod_{n=\infty}^{1}(1+X_n).$
	Since $$(1+X_{m})(1+X_{m-1})\dots(1+X_1)=1+\sum_{n=1}^m X_n(1+X_{n-1})(1+X_{n-2})\dots(1+X_1),$$ $\pa^!$ is then equal to the series 
	$$ \pa^!=\prod_{n=1}^{\infty}(1+X_n)=1+\sum_{n=1}^m X_n(1+X_{n-1})(1+X_{n-2})\dots(1+X_1).$$
Analogously, the $K$-dual of the language of partitions written in decreasing order, is the language of partitions with different parts, written in increasing order
\begin{equation}\label{eq.pertitiinsdecreasing}
\left(\prod_{n=\infty}^{1}\frac{1}{1-X_n}\right)^!=(1+\sum_{n=1}^{\infty}X_n\prod_{j=n}^1\frac{1}{1-X_n})^!=\prod_{n=1}^{\infty}(1+X_n)
\end{equation} 	
	 
\end{example}
 Let $L$ be a linked language $L=1+\A+L_B$. Given a subset  $\A_1$ of $\A$, and another set of links $C\subseteq \A_1\times\A$, let $N\subseteq\A_1\times\A^*$, be the language defined by
 $$N=\A_1+L_{C,B},$$
 where 
 $$L_{C,B}=\{\om|(\om_1,\om_2)\in C\mbox{ and } (\om_i,\om_{i+1})\in B,\; i=2,3,\dots,\ell(\om)-1\}.$$
 The language $N$ will be called a right (linked) $L$-module.
 Denote by $N^!$ (called the $K$-dual of $N$)  the $L^!$-module defined by
 $$N^!=\A_1+L_{C^c,B^c},$$
 where the complement of $C$ is taken over the set $\A_1\times \A$, $C^c=\A_1\times \A-C$.
 \begin{proposition}\label{prop.kdual2}
 	The generating function for the language $N^!$ defined as above is given by the formula
 	\begin{equation}\label{eq.rightmodule}
 	N^!=N^g L^!=N^g(L^g)^{-1},
 	\end{equation}
 	\noindent where the graded generating function $N^g$ is defined as
 	\begin{equation}
 	N^g=\sum_{\om\in N}(-1)^{\ell(\om)-1}\om.
 	\end{equation}
 \end{proposition}
\begin{proof}
We have to prove that $N^g L^!=\sum_{\om\in N^!}\om.$
We have 
$$N^g L^!=\sum_{(\om,\om')\in N\times L^!}(-1)^{\ell(\om)-1}\om\om'.$$
Define the function $\psi:N\times L^!\ra N\times L^!$ by considering the following cases.
Assume first that  $\ell(\om)\geq 2$.  If $(\om_k,\om'_1)\in B$  we make
$$\psi(\om_1\om_2\dots\om_k,\om'_1\om'_2\dots\om'_j)=(\om_1\om_2\dots\om_k\om'_1,\om'_2\dots\om'_j),$$  If $(\om_k,\om'_1)\in B^c$ or if $\om'= 1,$ define 
$$\psi(\om_1\om_2\dots\om_k,\om'_1\om'_2\dots\om'_j)=(\om_1\om_2\dots\om_{k-1},\om_k\om'_1,\om'_2\dots\om'_j).$$
Assume now that $\ell(\om)=1$. If $(\om_1,\om'_1)\in C$ define
$$\psi(\om_1,\om'_1\om'_2\dots\om'_j)=(\om_1\om'_1,\om'_2\om'_3\dots\om'_j).$$ Otherwise, if $(\om_1, \omega_1')\in  C^c$ or if  $\om'=1$ we make
$\psi(\om_1,\om')=\psi(\om_1,\om')$. The function $\psi$ is a sign reversing involution, its fixed points being of the form $(\om_1,\om')$, if either $(\om_1,\om_1')\in C^c$ or $\om'=1$. Then  
$$N^g L^!=\sum_{(\om_1,1)\in \A_1\times{1}}\om_11+\sum_{(\om_1,\om'):(\om_1,\om')\in C^c,\,\om'\in L^!-\{1\}}\om_1\om'_1=\A_1+L_{C^c,B^c}=N^!.$$
\end{proof}

\section{Shift and the shift plethystic trees language} Consider the algebra $\K\la\la\X\ran\ran$, $\X$ being the alphabet $$\X=\{X_0,X_1,X_2\dots\}.$$
Let $\kap=(\kap_1,\kap_2,\dots,\kap_m)$ be an element of $\N^m$ ( a weak composition). We denote by $X_{\kap}$ the word $X_{\kap_1}X_{\kap_2}\dots X_{\kap_n}$. As usual, the empty word will be denoted by $1$. We denote by $|\kap|$ the sum of its parts, $$|\kap|=\kap_1+\kap_2+\dots.$$

Let $R$ be an element of $\K\la\la\X\ran\ran$. The series $R$ is written as
$$R=\sum_{\kap\in \N^*}\la R, X_{\kap}\ran X_{\kap}.$$
\begin{remark}\normalfont{Let $\mathscr{S}$ be set of weak compositions. In the rest of the article, when no risk of confusion, we  identify $\mathscr{S}$ with the associated language $\{X_{\kap}|\kap\in \mathscr{S}\}$, and its generating series 
	$\sum_{\kap\in \mathscr{S}}X_{\kap}.$}
\end{remark}

We shall call $\kap$ a (strong) composition if $\kap_i\neq 0$, for every $i$. In what follows,  word `composition' will mean by defect {\em \textcolor{blue}{strong composition}}.

\begin{definition}
 Define $\sigma:\K\la\la\X\ran\ran\ra\K\la\la\X\ran\ran$ by extending the shift   $\sigma X_i=X_{i+1}$, $i=0,1,2\dots$, as a continuous algebra map. Equivalently, by making it multiplicative and to commute with the series sum symbol.
 \end{definition}
 
\begin{definition}A {\em \textcolor{blue}{shift plethystic}} (SP) tree is a plane rooted tree whose vertices are colored with colors in $\N$. The color of a given vertex indicates its height (length of the path from the root). \end{definition}
Let $T$ be a shift plethystic tree. We associate to $T$ the word $\om_T$ on $\X$, obtained by reading the vertices of $T$ in preorder from left to right as follows.
Assume that $T$ consist of $k\geq 0$ sub-trees, $T_1, T_2,\dots,T_k$, attached to The root (of color $0$). The preorder word of $T$ is then defined recursively  by
\begin{equation}\label{eq.recpre}\om_T=\begin{cases}X_0&\mbox{ if $k=0$} \\
X_0\,\sigma\om_{T_1}\,\sigma\om_{T_2}\dots\sigma\om_{T_k}&\mbox{ if $k> 0$.}\end{cases}
\end{equation} 
We denote by $\msA$ the language of shift plethystic trees,
\begin{equation}
\msA=\sum_{T}\om_T
\end{equation} 
\begin{figure}
	\begin{center}\includegraphics[width=70mm]{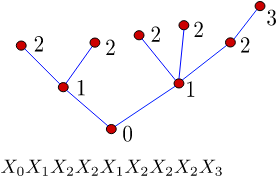}
	\end{center}\caption{Shift plethystic tree and associated word.}\label{fig.plethystictree}
\end{figure}
It is easy to check that the tree  $T$ is uniquely obtained from its word $\om_T$. 
The series $\sigma \msA$ gives us the language of shift plethystic trees with the root colored with color $1$, and every vertex colored with its height plus $1$. Similarly, $\sigma^n\msA$ is the language of shift plethystic trees, the root colored $n$ and each vertex colored $n$ plus its height.
\begin{theorem}
	The language $\msA$ can be expanded as the noncommutative  continued fraction
	\begin{equation}
	\msA=X_0\cfrac{1}{1-X_1\cfrac{1}{1-X_2\cfrac{1}{\ddots}}}=\lim_{n\ra\infty}X_0\cfrac{1}{1-X_1\cfrac{1}{1-X_2\cfrac{1}{\ddots 1- X_{n-1}\cfrac{1}{{1-X_n}}}}}
	\end{equation}
	\noindent
\end{theorem}
\begin{proof}Assume that the root of an SP tree has $k$ children, $k\geq 0$. By the definition of $\om_T$ (Eq. (\ref{eq.recpre})), to read its colors we read first the root and then read in preorder from left to right the colors of each one (or none) of the $k$ trees above the root. Each of them will produce a word in $\sigma\msA$. Hence  we have the identity
	\begin{equation}\label{eq.rectree}\msA=X_0(1+\sigma\msA+(\sigma\msA)^2+(\sigma\msA)^3+\dots)=X_0\frac{1}{1-\sigma\msA}\end{equation}
	Applying $\sigma^{j-1}$, $j=1,2,\dots,$ to both sides of the above identity we get $$\sigma^{j-1}\msA=X_{j-1}\frac{1}{1-\sigma^j\msA}.$$
	Recursively from Eq. (\ref{eq.rectree}), we obtain
	\begin{equation}\label{eq.rectree1}\msA=X_0\cfrac{1}{1-X_1\cfrac{1}{1-X_2\cfrac{1}{\ddots 1- X_{n-1}\cfrac{1}{{1-\sigma^n\msA}}}}}.\end{equation}
	Denote by $\msA_n$ the language $\msA$ restricted to the symbols $\{X_0,X_1,\dots,X_n\}$ (the words of SP trees of height at most $n$). We have that $\lim_{n\ra\infty}\msA_n=\msA$ and since $\sigma^n\msA_n=X_n$, from Eq. (\ref{eq.rectree1}) we obtain the result.  	
\end{proof}

\begin{proposition}\label{prop.tree-comp}\normalfont{
The words coming from shift plethystic trees are completely  characterized by the following properties:\begin{enumerate} \item Its first letter is $X_0$.
	\item If $\ell(\om_T)>1$ it is  followed by a word of the form $X_{\kap}$, $\kap$ being a composition with first element equal to $1$ and risings at most  $1$, $\kap_{i+1}-\kap_{i}\leq 1$. 
\end{enumerate}}
\end{proposition}
\begin{proof}
	Easy from Eq. (\ref{eq.recpre}), by induction on the number of vertices.
\end{proof}
\begin{definition}We denote by $\C$ the language of compositions, and  by $\C^{(1)}$ language of compositions with risings at most $1$. Observe that  $\sigma\C^{
(1)}$ consists of the compositions in $\C^{(1)}$, but where every part is at least $2$. More generally we define $\C^
{(m)}$ to be the language of compositions with rising at most $m$.\end{definition}
All the languages $\C$ and $\C^{(m)}$, $m\geq 1$, include the empty word.
	Proposition \ref{prop.tree-comp} can be now restated as follows, in terms of generating functions. \begin{proposition}\normalfont{The language $\msA$ can be expanded as
\begin{equation}\label{eq.splanguage}
\msA=X_0(1+\sum_{\kap\in\C^{(1)},\,\kap_1=1}X_{\kap}).
\end{equation}}
\end{proposition}
Define the alphabets
$$\X_+=\{X_1,X_2,X_3,\dots\}\mbox{ and }\X_{+2}=\{X_2,X_3,\dots\}.$$
\begin{definition}  We denote by $N$ the $\C^{(1)}$-module of compositions $\kap$ such that $\kap_1\geq 2$.
	\begin{equation}
N=\sum_{\kap\in\C^{(2)},\;\kap_1\geq 2}X_{\kap},
		\end{equation}
		\noindent 
The languages $\C^{(1)}$ and $\sigma\C^{(1)}$  are both linked. The language $\C^{(1)}\subset \X_+^*$ with set of links
\begin{equation*}
B=\{(X_i,X_j)|j-i\leq 1\}\subset \X_+\times\X_+, 
\end{equation*}
and $\sigma\C^{(1)}\subset\X_{2+}^*$ with the shifted set of links
\begin{equation*}
\sh B=\{(X_i,X_j)|j-i\leq 1,\; i,j\geq 2\}\subset \X_{2+}\times\X_{2+}.
\end{equation*}
The $\C^{(1)}$-module $N$ has as set of links $C\subset\X_2\times\X_+$,
\begin{equation*}
C=\{(X_i,X_j)|j-i\leq 1,\; i\geq 2\}\subset \X_{2+}\times\X_{+}.
\end{equation*}
\begin{definition}
	We denote by $\Lb_m$  the language of $m$-distinct partitions  (in the terminology of \cite{Andrews2004}). Being more explicit, $\Lb_2$ is the language of words of the form $X_{\lambda}$, $\lam$ being a partition (written in increasing order), $1\leq\lam_1<\lam_2<\lam_3,\dots$, satisfying \begin{equation}\label{eq.rest}\lambda_{i+1}-\lambda_i\geq m,\end{equation}
	(the empty and the singleton words being included in $\Lb_m$). In particular we have that $\Lb_0=\pa$ is the language of partitions with repetitions, $\Lb_1$, that of partitions without repetitions, and finally $\Lb_2$ is the language of $2$-distinct partitions, directly related to the combinatorics of the Rogers-Ramanujan identities. Observe that $\sigma\Lb_2$ is the language of $2$-distinct partitions, where each part is at least $2$.
\end{definition}
\begin{proposition}\normalfont{
We have that $\Lb_2$ is the $K$-dual of $\C^{(1)}$, $\sigma\Lb_2$ is the $K$- dual of $\sigma\C^{(1)}$. The $K$-dual of $N$ is the $\Lb_2$-module $\Nb$ of $2$-distinct partitions with first part greater than $2$, 
 \begin{eqnarray*}
\Lb_2&=&(\C^{(1)})^!\\
\sigma\Lb_2&=& (\sigma\C^{(1)})^!\\
\Nb&=&N^!
\end{eqnarray*}  }
\end{proposition}
\begin{proof} Easy, by simple inspection.	
\end{proof}
	Observe that since $\lam$ is a partition, if $\lam_1>1$, the rest of parts have also to be  greater than $1$, and the series $\Nb$ equals the non constant part of $\sh\Lb_2$,
	
	$$\Nb=\sh\Lb_2-1.$$
	Their graded generating functions are related as follows 
	\begin{equation}\label{Eq.graded}\Lb_2^g=1+\sum_{\lam_{i+1}-\lam_{i}\geq 2}(-1)^{\ell(\lam)}X_{\lam}=1-\sum_{\lam_{i+1}-\lam_{i}\geq 2}(-1)^{\ell(\lam)-1}X_{\lam}=1-\Nb^g.\end{equation}
\end{definition}

\begin{theorem}\label{th.quotient}
	The language $\msA$ can be expressed as the product
	\begin{equation}\label{eq.plethyRR}\msA= X_0(\sh\Lb_2^g)(\Lb_2^g)^{-1}.
	\end{equation}

\end{theorem}
\begin{proof} By Eq. (\ref{eq.splanguage}) we have
\begin{equation*}	
\msA=X_0(\C^{(1)}-N).
\end{equation*}
Since the operation of taking duals is involutive, $\Lb_2^!=\C^{(1)}$, $(\sh\Lb_2)^!=\sh\C^{(1)}$, and $\Nb^!=N$. By Eq. (\ref{Eq.graded}), Prop. \ref{prop.kdual2} and Prop. \ref{prop.kdual1},
$$\C^{(1)}-N=(\Lb^g_2)^{-1}-\Nb^g(\Lb^g_2)^{-1}=(1-\Nb^g)(\Lb_2^g)^{-1}=(\sh\Lb_2^g)(\Lb^g_2)^{-1}.$$

\end{proof}
Eq. (\ref{eq.plethyRR}) can be written more explicitly as the product
\begin{equation}\label{eq.arbRR1}
\msA=X_0(\sh\Lb_2^g)(\Lb_2^g)^{-1}=X_0(1+\sum_{\lam_{i+1}-\lam_i\geq 2,\,\lam_1\geq 2}(-1)^{\ell(\lam)}X_{\lam})(1+\sum_{\lam_{i+1}-\lam_i\geq 2,\, \lam_1\geq 1}(-1)^{\ell(\lam)}X_{\lam})^{-1}.
\end{equation}
Since $(\C^{(1)})^{-1}=\Lb_2^g$, and  $(\sh\C^{(1)})^{-1}=\sh\Lb_2^g$, Theorem \ref{th.quotient} has the following dual form
\begin{corollary}The language $\msA$ can be written as the product
	\begin{equation}\label{eq.arbRR2}
	\msA=X_0 (\sh \C^{(1)})^{-1}\C^{(1)}=X_0(1+\sum_{\kap_{i+1}-\kap_i\leq 1}X_{\kap})^{-1}(1+\sum_{\kap_{i+1}-\kap_i\leq 1,\, \kap_i\geq 2}X_{\kap}).
	\end{equation}
	\end{corollary}

\section{Path length and $q$-series}
For a series $S$ on the alphabet $\X$, making the \textcolor{blue}{ $q$-umbral evaluation $X_k\rightarrow zq^k$} in the pair of commuting variables $z$ and $q$, we obtain a $q$-series that by abuse of language we denote with the same symbol $S$, $S(z,q)$. Observe that for every series $S$ in $\K\la \la\X\ran\ran$ we have \begin{equation}\label{eq.shiftq}(\sigma S)(z,q)=S(zq,q).\end{equation} Recall that the \textcolor{blue}{\em{path length}} of a rooted tree is defined to be the sum of the heights of its vertices. When we make the substitution $X_k\rightarrow zq^k$ in the word $\om_T$ associated to a tree $T$ we get $z^n q^{\mathrm{pl(T)}}$, where $n$ is the number of vertices of $T$ and $\mathrm{pl}(T)$ its path length. For example, in the tree of Fig. \ref{fig.plethystictree}, the $q$-substitution in the word $X_0X_1X_2X_2X_1X_2X_2X_2X_3$ gives us
$$X_0X_1X_2X_2X_1X_2X_2X_2X_3\mapsto z(zq)(zq^2)^2(zq)(zq^2)^3(zq^3)=z^9q^{15}.$$

 Then, the $q$-series $\msA(z,q)$ counts the number of plane rooted trees according with their path length. Observe that the path length of a plane rooted tree with $n$ vertices is bounded by the path length of the branch-less tree, which is equal to $0+1+2+3+\dots +n-1=\binom{n}{2}$.  From Eq. (\ref{eq.rectree1}) we get  
	\begin{equation}
	\msA(z,q)=\sum_{n=1}^{\infty}(\sum_{m=0}^{\binom{n}{2}} P(n,m)q^m)z^n=\cfrac{z}{1-\cfrac{zq}{1-\cfrac{zq^2}{\ddots}}},
	\end{equation}
\noindent where $P(n,m)$ is the number of plane rooted tree on $n$ vertices having path length equal to $m$, 
\begin{eqnarray*}\msA(z,q)&=&z + q z^2 + (q^2 + q^3) z^3+ (q^3 + 2 q^4 + q^5 + q^6) z^4\\ &+& (q^4 + 
	3 q^5 + 3 q^6 + 3 q^7 + 2 q^8 + q^9 + q^{10}) z^5\\ &+& (q^5 + 4 q^6 + 
	6 q^7 + 7 q^8 + 7 q^9 + 5 q^{10} + 5 q^{11} + 3 q^{12} + 2 q^{13} + q^{14} +
	q^{15}) z^6+\dots.\end{eqnarray*}

From Eq. (\ref{eq.arbRR1}),
\begin{equation}\label{eq.arbrr1}
\msA(z,q)=z\frac{ \Lb_2(zq,z)}{\Lb_2(z,q)}=z\frac{1+\sum_{\lam\in\sh\Lb_2}(-z)^{\ell(\lam)}q^ {|\lam|}}{1+\sum_{\lam\in\Lb_2}(-z)^{\ell(\lam)}q^{|\lam|}}.
\end{equation}
From Eq. (\ref{eq.arbRR2}) we obtain the dual expression
\begin{equation}	
\msA(z,q)=z\frac{\C^{(1)}(z,q)}{\C^{(1)}(zq,q)}=z\frac{1+\sum_{\kap \in\C^{(1)}}z^{\ell(\kap)}q^{|\kap|}}{1+\sum_{\kap \in\sh\C^{(1)}}z^{\ell(\kap)}q^{|\kap|}}. \end{equation}

  \section{Rogers-Ramanujan Identities and Compositions}\label{sec.RRcompositions}
  \begin{theorem}\label{th.RRcompositions}
  	We have the following identities
  	\begin{eqnarray}
  (\C^{(1)})(-1,q)=	(\C^{(1)})^{g}(1,q)=1+\sum_{\kap\in\C^{(-1)}}(-1)^{\ell(\kap)}q^{|\kap|}&=&\prod_{k=0}^{\infty}(1-q^{5k+1})(1-q^{5k+4})\\ (\sh\C^{(1)})(-1,q)=(\sh\C^{(1)})^g(1,q)=1+\sum_{\kap\in\C_2}(-1)^{\ell(\kap)}q^{|\kap|}&=&\prod_{k=0}^{\infty}(1-q^{5k+2})(1-q^{5k+3})
  	\end{eqnarray}
  	\end{theorem}\begin{proof}
  From Proposition \ref{prop.kdual1}, $(\C^{(1)})^g=(\pa_2)^{-1}$ and $(\sh\C^{(1)})^g=(\sf\pa_2)^{-1}$. Then, $q$-umbral evaluation gives us 
  \begin{eqnarray*} (\C^{(1)})^{g}(z,q)=\sum_{\kap\in\C^{(1)}}(-1)^{\ell(\kap)}z^{\ell(\kap)}q^{|\kap|}&=&\frac{1}{1+\sum_{\lam\in\pa_2} z^{\ell(\lam)}q^{|\lam|}}\\
  	(\sh\C^{(1)})^{g}(z,q)=\sum_{\kap\in\sh\C^{(1)}}(-1)^{\ell(\kap)}z^{\ell(\kap)}q^{|\kap|}&=&\frac{1}{\sum_{\lam\in\sh\pa_2} z^{\ell(\lam)}q^{|\lam|}}\\
  	  \end{eqnarray*}
By the well known identities
\begin{eqnarray*}
\sum_{\lam\in\pa_2} z^{\ell(\lam)}q^{|\lam|}&=&\sum_{n=0}^{\infty}\frac{z^n q^{n^2}}{(1-q)(1-q^2)\dots (1-q^n)}\\
\sum_{\lam\in\sh\pa_2} z^{\ell(\lam)}q^{|\lam|}&=&\sum_{n=0}^{\infty}\frac{z^n q^{n(n+1)}}{(1-q)(1-q^2)\dots (1-q^n)},
\end{eqnarray*}
using the Rogers-Ramanujan identities (Eq. (\ref{eq.RR1}) and Eq. (\ref{eq.RR2})) we get the result.
 \end{proof} 
Let $\C^{(1)}[n]$ and $\C^{(1)}[n,k]$ respectively be the set of compositions of $n$ in $\C^{(1)}$, and the set of composititions of $n$ in $\C^{(1)}$  having exactly $k$ parts. Similarly, define $\sh\C^{(1)}[n]$ and $\sh\C^{(1)}[n,k]$. From Theorem \ref{th.RRcompositions}, we get the identities
\begin{eqnarray}\label{eq.signedpartitions1} 1+\sum_{n=1}^{\infty} (\sum_{k=1}^{n}(-1)^k |\C^{(1)}[n,k]| ) q^n&=&\prod_{k=0}^{\infty}(1-q^{5k+1})(1-q^{5k+4})\}\\\label{eq.signedpartitions2}1+\sum_{n=2}^{\infty} (\sum_{k=1}^{n}(-1)^k |\sh\C^{(1)}[n,k]| ) q^n&=&\prod_{k=0}^{\infty}(1-q^{5k+2})(1-q^{5k+3})
\end{eqnarray}
Observe that the series in the right hand side of equation (\ref{eq.signedpartitions1}) gives us the partitions (in decreasing order) with distinct parts congruent with $1$ or $4$ module $5$, signed by its number of parts. The right hand side of (\ref{eq.signedpartitions2}) enumerates a similar kind of signed partitions, each part congruent with $2$ or $3$ module $5$. 
 For example $4+1$ is the only partition of $5$ enumerated by the right hand side of Eq. (\ref{eq.signedpartitions1}), and $\{7+3,\; 8+2\}$  are the only partitions of $10$ enumerated by the right hand side of Eq. (\ref{eq.signedpartitions2}). 
The compositions in $\sh\C^{(1)}(10)$ and in $\sh\C^{(1)}(11)$ are given respectively in tables \ref{tab.comp1} and \ref{tab.comp2}.
\begin{table}\begin{center}
\begin{tabular}{|l|l|l|l|l|l|l|l|c|}
	\hline 
$k$	&  &  &  &  &  &  &  & weight  \\ 
	\hline 
1	&10  &  &  &  &  &  &  & \textcolor{red}{-1} \\ 
	\hline 
2	& $55$ & $64$ &$\textcolor{blue}{73}$  &$\textcolor{blue}{82}$  &&&&\textcolor{blue}{2}  \\ 
	\hline 
3	&$532$  &$523$&$622$  &$442$  &$433$  &343 &334  &\textcolor{red}{-7}  \\ 
	\hline 
4	&$2233$  & $2323$ &$3223$  &$2332$  &$3232$  & $3322$ &$4222$ & \textcolor{blue}{7}  \\ 
	\hline 
5	& $22222$ &  &&  &&  && \textcolor{red}{-1} \\ 
	\hline
\end{tabular}\caption{Compositions in $\sh\C^{(1)}[10]$}\label{tab.comp1}\end{center}\end{table}

\begin{table}\begin{center}
\begin{tabular}{|l|l|l|l|l|l|l|l|c|}
	\hline 
	$k$	&  &  &  &  &  &  &  & weight  \\ 
	\hline 
	1	&11  &  &  &  &  &  &  & \textcolor{red}{-1} \\ 
	\hline 
	2	& $56$ & $65$ &74  &$\textcolor{blue}{83}$  &92&&&\textcolor{blue}{4}  \\ 
	\hline 
	
	3	&$722$&$632$&$623$&$542$ &$533$ &$452$&$443$&  \textcolor{red}{-9}\\  
	\hline &$434$&$344$&&&&&&\\\hline   
	4&$522$&$4322$&$4232$&$4223$&$3422$&$3332$&$3323$&\textcolor{blue}{11}\\\hline&$3233$&$2333$&$2342$&$2234$   
	&&&&\\\hline
	5	&$32222$&  $23222$ &$22322$  &$22232$&$22223$ &&& \textcolor{red}{-5} \\ 
	\hline 
\end{tabular}\caption{Compositions in $\sh\C^{(1)}[11]$}\label{tab.comp2}\end{center}\end{table}
Consider the set $\widehat{\C^{(1)}}[n,k]$ of restricted compositions. This set only excludes from $\C^{(1)}[n,k]$ the strictly decreasing compositions with all its parts congruent with $1$ or $4$ module $5$.  In a similar way, we define $\widehat{\sh\C^{(1)}}[n,k]$. 

 Then, the Rogers-Ramanujan identities have the following combinatorial form in terms of compositions.
\begin{theorem}The signed sets $\widehat{\C^{(1)}}[n]$ and $\widehat{\sh\C^{(1)}}[n]$ have both zero total weight 
	\begin{eqnarray}
	&&\sum_{k=1}^n(-1)^k|\widehat{\C^{(1)}}[n,k]|=0\mbox{, for $n\geq 1$}\\
	&&\sum_{k=1}^n(-1)^k |\widehat{\sh\C^{(1)}}(n,k)|=0\mbox{, for $n\geq 2$}.
	\end{eqnarray}
\end{theorem}
 \section{Shift plethysm, and general shift-plethystic trees}\label{sec.splety}
 \begin{definition} Let $R$ be a series in $\K\laa \X\raa$ with  zero constant term, $\la R,1\ran=0$.
 	We define the  \textcolor{blue}{{\em shift-plethystic substitution}} of $R$ in a word $X_\kap=X_{\kap_1}X_{\kap_2}X_{\kap_3}\dots X_{\kap_l}$, as the substitution of the shift $\sigma^{\kap_i}R$ on each of the letters of $X_{\kap}$,
 	\begin{equation*}
 	X_{\kap}\spl R:=(\sigma^{\kap_1}R)(\sigma^{\kap_2}R)\dots(\sigma^{\kap_l}R).
 	\end{equation*}
 	For a formal power series $T$, define the shift plethysm $T\spl R$ by 
 	\begin{equation}\label{eq.shiftplethys}
 	T\spl R=\sum_{\kap\in\N^*}\la T,X_{\kap}\rangle X_{\kap}\spl R=\sum_{\kap\in\N^*}\la T,X_{\kap}\rangle  (\sigma^{\kap_1}R)(\sigma^{\kap_2}R)\dots(\sigma^{\kap_l}R).
 	\end{equation}
 \end{definition}
The series in the right hand side of Eq. (\ref{eq.shiftplethys}) is convergent. To see this, let us denote by $R^{\kap}$ the product $(\sigma^{\kap_1}R)(\sigma^{\kap_2}R)\dots(\sigma^{\kap_l}R)$. We see that $\la R^{\kap},X_{\tau}\ran=0$ whenever either $l=\ell(\kap)>|\tau|$ or $|\kap|>|\tau|$. Hence the set $\{\tau|\la R^{\kap},X_{\tau}\ran\neq 0\}$ is finite. Shift plethysm is an associative operation having $X_0$ as identity. 
\begin{proposition}\normalfont Every series $R$ with zero constant term and such that $\la R,X_0\rangle\neq 0$ has a two sided shift plethystic inverse, denoted $R^{\la-1\ran}$,
	$$R\spl R^{\la-1\ran}=R^{\la-1\ran}\spl R=X_0$$
	\end{proposition} 
\begin{proof}Let  $\alpha\neq 0$ be the value of  $R$ at $X_0$, $\la R,X_0\rangle=\alpha\neq 0$. Define $R_+=R-\alpha X_0$, and the series $T$ by the implicit equation
	\begin{equation*}
	T=\alpha^{-1}(X_0-R_+\spl T).
	\end{equation*}
From here we get $\alpha T+R_+\spl T=X_0$. Which can be written as $(\alpha X_0+R_+)\spl T=R\spl T=X_0$. Then, $R^{\la -1\ran}=T$.	
\end{proof}
\subsection{Shift-plethysm and $q$-composition of series}\label{sec.pletitrees}
In this subsection we show how shift-plethysm generalizes the classical definition of $q$-composition of series. This is relevant due to the importance of the $q$-Lagrange inversions formulas for applications in proving identities in $q$-series (see \cite{andrews1975qlagrange}, \cite{gessel1980noncommutative} \cite{garsia1981qLag}, \cite{gessel1983applications}, \cite{garsia1986novel},  \cite{krattenthaler1988qLag}). A general shift-plethystic Lagrange inversion, not yet found, would lead to new forms of $q$-Lagrange inversion as well as to the reinterpretation in a general context of the known ones. Shift-plethysm also offers the advantage, in contrast to $q$-substitution, of being an associative operation. From that,  the plethystic inverse is a bilateral one, also in contrast to the known forms of  $q$-composition inverse.  \\
 Let $C$ be the series  $$C=\sum_{n=1}^{\infty}c_nX_0X_1X_2\dots X_{n-1}.$$ Consider the shift plethysm $H:=C\spl R$, $R$ being an arbitrary series with zero constant term. We have  
 \begin{equation}\label{eq.qandshiftplethys}
H= C\spl R=\sum_{n=1}^{\infty}c_n  R(\sigma R)(\sigma^2R)\dots (\sigma^{n-1}R).
 \end{equation}
 Taking $q$-series, by Eq. (\ref{eq.shiftq}) we recover the classical $q$-substitution,
 \begin{equation}
H(z,q)= (C\spl R)(z,q)=\sum_{n=1}^{\infty}c_n  R(z,q)R(zq,q)R(zq^2,q)\dots R(zq^{n-1},q).
 \end{equation}

Now consider $R$ to be a series in the variable $X_0$, and express it in the form $R(X_0)=X_0\phi^{-1}(X_0)$. Shift plethysm with $C$ will give us
\begin{equation}
H=\sum_{n=1}^{\infty}c_n\frac{X_0}{\phi(X_0)}\frac{X_1}{\phi(X_1)}\dots\frac{X_{n-1}}{\phi(X_{n-1})}.
\end{equation} 
 Which, by $q$-umbral evaluation, gives
 \begin{equation}
H(z,q)=\sum_{n=1}^{\infty}c_n\frac{q^{\binom{n}{2}}z^n}{\phi(z)\phi(zq)\dots\phi(zq^{n-1})}.
 \end{equation} 
 Obtaining $c_n$ in terms of the $h_n$ in the expansion of $H(z,q)$ is similar to the $q$-Lagrange inversion problem in \cite{andrews1975qlagrange}.
 \subsection{Enriched shift-plethystic trees}
 In this section we introduce the $M$-enriched shift-plethystic trees, $M$ being a normalized invertible (non-commutative) series, based in the similar notion formalized by Joyal in \cite{joyal1981theorie} and its plethystic generalization in the commutative framework of colored species \cite{Mendezava}.  

 \begin{definition} Let $M$ be a series with constant term equal to $1$, $\la M,1\ran=1$. We define the $M$-enriched trees series by the implicit equation
 \begin{equation}\label{eq.shiftrees}
 \msA_M=X_0(M\spl \msA_M).
 \end{equation}\end{definition}
 \begin{proposition}\label{prop.treeinverse}The shift-plethystic inverse of $\msA_M$ is given by the formula
 \begin{equation*}(\msA_M)^{\la -1\ran}=X_0M^{-1}\end{equation*}
 \end{proposition}
 \begin{proof}
 From Eq. (\ref{eq.shiftrees}) we have $\msA_M(M^{-1}\spl \msA_M)=(X_0 M^{-1})\spl\msA_M=X_0.$
 \end{proof}

The plethystic inversion for enriched trees, in the most elementary examples where the shift-plethystic inverse can be easily computed, leads by $q$-umbral evaluation to generalizations of some classical formulas.\\
\begin{example}\label{ex.shiftplethy}
	From Eq. (\ref{eq.rectree}), the SP trees $\msA$ satisfy the implicit equation
	\begin{equation}
	\msA=X_0(\frac{1}{1-X_1}\spl\msA).
	\end{equation}
	 Hence, they are obtained by enriching with the series $\frac{1}{1-X_1}$,
	$\msA=\msA_{\frac{1}{1-X_1}}$. From this we get
	
	\begin{equation*}
	\msA(1-\sigma\msA)=\msA-\msA\sigma\msA=(X_0-X_0X_1)\spl\msA=X_0,\end{equation*}
	its shift plethystic inverse 
	\begin{equation*}
	\msA^{\la -1\ran}=X_0-X_0X_1,
	\end{equation*}
	and the implicit equation
	\begin{equation*}
	\msA=X_0+(X_0X_1)\spl \msA.
	\end{equation*}
	 The $q$-series of the SP trees satisfies the implicit equations 
	\begin{eqnarray*}
		\msA(z,q)&=&z+\msA(z,q)\msA(zq,q)\\
		\msA(z,q)&=&\frac{z}{1-\msA(zq,q)}.
	\end{eqnarray*}
Those equations were studied by Garsia in \cite{garsia1981qLag} in relation with his $q$-Lagrange inversion formulas, but without any combinatorial interpretation.
\end{example}

	\begin{example} Let $\Lo$ be the language $$\Lo=1+X_0+X_0X_1+X_0X_1X_2+X_0X_1X_2X_3+\dots.$$ The language of the branchless trees, enriched with $1+X_1$ is equal to $\Lo_+=\Lo-1$,
	$$\Lo_+=\msA_{(1+X_1)}=X_0(1+\sigma\Lo_+).$$ Its shift-plethystic inverse is equal to $$\Lo_+^{\la -1\ran}=X_0\frac{1}{1+X_1}.$$ 
	Then,
	$$\Lo_+\spl (X_0\frac{1}{1+X_1})=\sum_{n=1}^{\infty}\prod_{j=1}^{n}X_{j-1}\frac{1}{1+X_{j}}=X_0.$$
	The $q$-umbral evaluation gives us
	$$\sum_{n=1}^{\infty}\frac{q^{\binom{n}{2}}z^n}{(1+zq)(1+zq^2)\dots(1+zq^n)}=z.$$
	
From that
$$\sum_{n=0}^{\infty}\frac{q^{\binom{n}{2}}z^n}{(1+zq)(1+zq^2)\dots(1+zq^n)}=1+z.$$
Making $z=1$ and $z=-1$ we recover respectively the classical identities $A.1$ and $A.4$ in \cite{sills2017invitation}.
 \end{example}
\begin{example}
	Let $\Lo^{(e)}_+$ be the even form of $\Lo_+$,
	\begin{equation*}\Lo^{(e)}_+=\sum_{n=0}^{\infty}X_0X_2X_4\dots X_{2n-2}=X_0(1+\sh^2\Lo_+^{(e)})=\msA_{(1+X_2)}.
	\end{equation*}
	Its shift-plethystic inverse is equal to \begin{equation*}
(\Lo^{(e)}_+)^{\la -1\ran}=X_0\frac{1}{1+X_2}.
	\end{equation*}
	Hence we have the identity
	\begin{equation}\label{eq.Rogers171}
	\Lo^{(e)}_+\spl X_0\frac{1}{1+X_2}=	\sum_{n=1}^{\infty}\prod_{j=0}^{n-1}X_{2j}\frac{1}{1+X_{2j+2}}=X_0.
	\end{equation}
	The odd version of $\Lo_+$, $\Lo^{(o)}_+$, is equal to the shift $\sh\Lo^{(e)}_+$. 
	$$\Lo^{(o)}_+=\sum_{n=1}^{\infty}X_1X_3X_5\dots X_{2n-1}.$$
	
	From Eq.(\ref{eq.Rogers171}), by shifting and adding $1$  we obtain \begin{equation*}
\Lo^{(o)}\spl X_0\frac{1}{1+X_2}=1+\Lo^{(o)}_+\spl X_0\frac{1}{1+X_2}=1+\sum_{n=1}^{\infty}\prod_{j=0}^{n-1}X_{2j+1}\frac{1}{1+X_{2j+3}}=1+X_1.
	\end{equation*}
	Multiplying by $(1+X_1)^{-1}$ the left of both sides of the rightmost equality
	$$\frac{1}{1+X_1}+\sum_{n=1}^{\infty}\frac{1}{1+X_1}\prod_{j=0}^{n-1}X_{2j+1}\frac{1}{1+X_{2j+3}}=1.$$
	By $q$-umbral evaluation we obtain
	\begin{equation}\label{eq.Rogers17}
	\sum_{n=0}^{\infty}\frac{q^{n^2}z^n}{(1+zq)(1+zq^3)\dots(1+zq^{2n+1})}=1.
	\end{equation}
	\noindent Eq. (\ref{eq.Rogers17}) generalizes Rogers identity (C) 6 in \cite[p.~333]{Rogers17}, obtained by specializing to $z=-1$. See also \cite{sills2017invitation}, Formula (A.2).\end{example}
	\begin{example}Denote by $\Lsl$ the language obtained from $\Lo_+$ by left shift plethysm with the series $
\Al_0=\sum_{j=0}^{\infty}X_j,$
	\begin{equation}\label{eq.sigmaL}
	\Lsl=\sum_{j=0}^{\infty}X_j\spl \Lo_+=\sum_{j=0}^{\infty}\sum_{n=1}^{\infty}X_jX_{j+1}\dots X_{n+j-1}.
	\end{equation}
	Since $\Al_0-\sigma\Al_0=X_0$, the shift-plethystic inverse of $\Al_0$ is equal to $X_0-X_1$. Hence
	\begin{equation*}(\Lsl)^{\la -1\ran}=(\Al_0\spl\Lo_+)^{\la -1\ran}=(\Lo_+)^{\la -1\ran}\spl(X_0-X_1)=(X_0-X_1)\frac{1}{1+(X_0-X_1)}.\end{equation*}
	By plethystic composition with $\Lsl$ we obtain the identity
	\begin{equation*}
	\sum_{j=0}^{\infty}\sum_{n=1}^{\infty}(X_j-X_{j+1})\frac{1}{1+(X_{j+1}-X_{j+2})}\dots (X_{n+j-1}-X_{n+j})\frac{1}{1+(X_{n+j}-X_{n+j+1})}=X_0.
	\end{equation*}
	Interchanging sums, by $q$-umbral evaluation,
	\begin{equation}
	\sum_{n=1}^{\infty}z^nq^{\binom{n}{2}}(1-q)^n\sum_{j=0}^{\infty}\frac{q^{jn}}{\prod_{k=1}^n({1-zq^{j+k}(1-q))}}=z.
	\end{equation}
	Equivalently, making the change $z(1-q)\mapsto z$, 
		\begin{equation}
	\sum_{n=1}^{\infty}z^nq^{\binom{n}{2}}\sum_{j=0}^{\infty}\frac{q^{jn}}{\prod_{k=1}^n({1-zq^{j+k})}}=\frac{z}{1-q}.
	\end{equation}
		\end{example}

\begin{example} 
	 The shift-plethystic trees enriched with the language $\sigma\Lo$, $\msA_{\sigma\Lo}$, satisfy the equation
	\begin{equation*}
	\msA_{\sigma\Lo}=X_0(1+\sigma\msA_{\sigma\Lo}+(\sigma\msA_{\sigma\Lo})(\sigma^2\msA_{\sigma\Lo})+\dots).
	\end{equation*}
	Its shift-plethystic inverse is equal to 
	\begin{equation*}
	\msA_{\sigma\Lo}^{\la -1\ran}=X_0(\sigma\Lo)^{-1}=X_0(1+X_1+X_1X_2+X_1X_2X_3+\dots)^{-1}.
	\end{equation*} 
\end{example}
\begin{example}
	 The series of shift-plethystic trees enriched with  $$M=(1-\sigma\Lo_+)^{-1}=(1-(X_1+X_1X_2+X_1X_2X_3+\dots))^{-1}$$
	  satisfies the implicit equations
	 \begin{eqnarray*}
	 	\msA_{M}&=&X_0\frac{1}{1-\sigma\Lo_+\spl \msA_{M}}\\\msA_{M}&=&X_0+(\msA_M)(\sigma\msA_M)+(\msA_M)(\sigma\msA_M)(\sigma^2\msA_M)+\dots.
	 		\end{eqnarray*}
	 Its shift-plethystic inverse is equal to
	 $$\msA_M^{\la -1\ran}=X_0-X_0\sigma \Lo_+.$$
	 Taking $q$-series we obtain the implicit equation
	 $$\msA_M(z,q)=z+\msA_M(z,q)\msA_M(zq,q)+\msA_M(z,q)\msA_M(zq,q)\msA_M(zq^2,q)+\dots$$
	 \end{example}
\section{Some other shift-plethystic identities}\label{sec.RRandnew} In this section we establish some relations between the languages of partitions, compositions and shifted plethystic trees.
As a motivating example for Theorem \ref{th.plethypartition}, let us take the following composition in $\C^{(1)},$
$$\kap=56763454343342332.$$ Placing a bar before each local (non-strict) minimum of the sequence,
$$\textcolor{red}{|}5676\textcolor{red}{|}3454\textcolor{red}{|}34\textcolor{red}{|}3\textcolor{red}{|}34\textcolor{red}{|}233\textcolor{red}{|}2.$$ We see that the local minima form a partition in weakly decreasing form $\lam=5333322$. Each word between two bars is associated to a word of a shifted plethystic tree, $$X_{5676}X_{3454}X_{34}X_{3}X_{34}X_{233}X_2$$ is in the language $(\sh^5\msA)(\sh^3\msA)(\sh^3\msA)(\sh^3\msA)(\sh^3\msA)(\sh^2\msA)(\sh^2\msA)=X_{\lam}\spl\msA.$

\begin{theorem}\label{th.plethypartition}We have the following identities
	\begin{eqnarray}\label{eq.l1identity}
	\C^{(1)}&=&\prod_{n=\infty}^{1}\frac{1}{1-X_n}\spl\msA=\prod_{n=\infty}^{1}\frac{1}{1-\sh^n\msA}.\\ \label{eq.l2identity}\sh\C^{(1)}&=&\prod_{n=\infty}^{2}\frac{1}{1-X_n}\spl\msA=\prod_{n=\infty}^{2}\frac{1}{1-\sh^n\msA}.
	\end{eqnarray}
	
\end{theorem}
\begin{proof}
	Let $\kap$ be a composition in $\C^{(1)}$. Define $i_1=1$ and $\lam_1=\kap_1$, and while the set $A_{r-1}=\{i>i_{r-1}|\kap_i\leq\kap_{i_{r-1}}=\lam_{r-1}\}$ is nonempty define recursively
	$$i_r=\mbox{min}A_{r-1}\mbox{ and }\lam_{r}=\kap_{i_r}.$$ 
	Let $\kap^{(r)}$ be the segment of $\kap$ after (and including) $\lam_r=\kap_{i_r}$ and before (and excluding) $k_{i_{r+1}}=\lam_{r+1}$. We claim that each word $X_{\kap^{(r)}}$ is in the language $\sh^{\lam_r}\msA$. If $\ell(\kap^{(r)})=1$ the statement is trivial. If  $\ell(\kap^{(r)})>1$, it follows since $\kap^{(r)}$ is in $\C^{(1)}$ and all of its parts after $\lam_r$ (the shifted height of the root) are greater  than it.  Hence, for each word $X_{\kap}\in \C^{(1)}$, there exists a unique partition $\lam_1\geq\lam_2\geq\dots\geq\lam_l,$ as defined above, such that $X_{\kap}\in X_{\lam}\spl\msA$. Conversely, every word in $ X_{\lam}\spl\msA$ is in $\C^{(1)}$.  Then $\C^{(1)}$ can be expanded as follows, using the generating function of the weakly decreasing partitions (\ref{eq.pertitiinsdecreasing}) 
	\begin{equation*}
	\C^{(1)}=\sum_{\lam}X_{\lam}\spl\msA=\prod_{n=\infty}^1\frac{1}{1-X_n}\spl\msA.
	\end{equation*}
	Eq. (\ref{eq.l2identity}) follows immediately by shifting.
	
\end{proof}
From Eq.(\ref{eq.l1identity}) we get
 \begin{equation}
(\C^{(1)})^{g}=(\prod_{n=\infty}^1\frac{1}{1-X_n})\spl\msA(-\X)
 \end{equation}
 \noindent where $\msA(-\X)$ is the graded generating function of $\msA$. Using the fact that the shift-plethystic inverses of $\msA$ and $\msA(-\X)$ are respectively equal to $X_0-X_0X_1$ and $X_0X_1-X_0$, we get the identities
 \begin{eqnarray}\label{eq.inverseapp1}
 \C^{(1)}\spl (X_0-X_0X_1)&=&\prod_{n=\infty}^1\frac{1}{1-X_n}\\\label{eq.inverseapp2}
 \Lb_2\spl(X_0X_1-X_0)&=&((\C^{(1)})^{g})^{-1}\spl (X_0X_1-X_0)=\prod_{n=1}^{\infty}(1-X_n)
 \end{eqnarray}
 The left hand side of equations (\ref{eq.inverseapp1}) and (\ref{eq.inverseapp2}) are respectively equal to  
  \begin{eqnarray*}
 \C^{(1)}\spl (X_0-X_0X_1)&=&\sum_{\kap\in\C^{(1)}}\prod_{i=1}^{\ell(\kap)}X_{\kap_i}(1-X_{\kap_i+1})\\
 \Lb_2\spl(X_0X_1-X_0)&=&\sum_{\lam\in\Lb_2}\prod_{i=1}^{\ell(\lam)}X_{\lam_i}(X_{\lam_i+1}-1)\\
 \end{eqnarray*}
\noindent Substituting   $X_{n}\mapsto q^n,$ we get the identities
 \begin{eqnarray*}
	\sum_{n=0}^{\infty}q^n\sum_{\kap\in\C^{(1)}[n]}\prod_{i=1}^{\ell(\kap)}(1-q^{\kap_i+1})&=&\prod_{n=1}^{\infty}\frac{1}{1-q^n}\\
	\sum_{n=0}^{\infty}q^n\sum_{\lam\in\Lb_2[n]}\prod_{i=1}^{\ell(\lam)}(q^{\lam_i+1}-1)&=&\prod_{n=1}^{\infty}{(1-q^n)}.
\end{eqnarray*}


\bibliographystyle{amsplain}

\end{document}